\documentclass[sn-mathphys-num]{sn-jnl}% Math and Physical Sciences Numbered Reference Style
\usepackage{graphicx}%
\usepackage{multirow}%
\usepackage{amsmath,amssymb,amsfonts}%
\usepackage{amsthm}%
\usepackage{mathrsfs}%
\usepackage[title]{appendix}%
\usepackage{xcolor}%
\usepackage{textcomp}%
\usepackage{manyfoot}%
\usepackage{booktabs}%
\usepackage{algorithm}%
\usepackage{algorithmicx}%
\usepackage{algpseudocode}%
\usepackage{listings}%
\theoremstyle{thmstyleone}%
\newtheorem{theorem}{Theorem}%  meant for continuous numbers
\newtheorem{proposition}[theorem]{Proposition}%

\theoremstyle{thmstyletwo}%
\newtheorem{example}{Example}%
\newtheorem{remark}{Remark}%

\theoremstyle{thmstylethree}%
\newtheorem{definition}{Definition}
\newtheorem{lemma}{Lemma}
\newtheorem{corollary}{Corollary}
 \numberwithin{equation}{section}

\begin{document}

\title[]{Novel Perturbed b-Metric and Perturbed Extended b-Metric Spaces with Banach-Type Fixed Point Theorem}
%{Fixed Point Approach for Perturbed Parametric Metric Space with Applications to the Fixed-Circle Problem}

\author[]{\fnm{Abdelhamid} \sur{Moussaoui}}\email{a.moussaoui@usms.ma}

%\author[2]{\fnm{Nihal} \sur{Ta\c{s}}}\email{nihaltas@balikesir.edu.tr}

%\author[2]{\fnm{Stojan} \sur{Radenovi}}\email{radens@beotel.rs}

%\author[1]{\fnm{Said} \sur{Melliani}}\email{s.melliani@usms.ma}

\affil[]{\orgdiv{Laboratory of Applied Mathematics $\&$ Scientific Computing}, \orgname{Faculty of Sciences and Technics, Sultan Moulay Slimane University, \city{Beni Mellal},, \country{Morocco}}}
%\affil[2]{\orgdiv{Bal\i kesir University,
%}\orgname{Department of Mathematics, 10145 Bal\i
%kesir, T\"{u}rkiye}}

\abstract{In this paper, we introduce a new general framework, called \emph{perturbed 
extended $b$-metric spaces}, denoted by $(X,\mathcal{D}_{\zeta},\hbar)$, which 
extends the classical and extended $b$-metric structures through the inclusion 
of an explicit perturbation mapping $\hbar$. This formulation is motivated by 
the observation that distance measurements in many analytical and applied 
contexts are often affected by intrinsic or external deviations that cannot 
be captured by the usual metric-type geometries. We also identify a 
meaningful specialization arising when the control function $\zeta$ is 
constant, leading to the notion of a \emph{perturbed $b$-metric space}, 
introduced here as a natural restriction of the general framework.

We establish several fundamental properties of spaces of the form 
$(X,\mathcal{D}_{\zeta},\hbar)$ and develop a Banach-type fixed point theorem 
in the perturbed extended $b$-metric setting. Conditions ensuring the 
existence and uniqueness of fixed points of a self-map $T:X\to X$ are 
derived, together with convergence of Picard iterations $T^{n}v \to \vartheta$. 
Illustrative examples are provided to show that the presence of the 
perturbation term $\hbar$, together with the influence of the control function 
$\zeta$, may cause $\mathcal{D}_{\zeta}$ to lose the usual extended $b$-metric 
behaviour and prevent $\mathcal{D}_{\zeta}$ from satisfying the standard 
extended $b$-metric axioms. We further discuss natural extensions of this perturbation philosophy to 
multi-point settings, encompassing $S$-metric spaces and their extended and $S_b$-metric variants.
The results presented here open new directions for the study of contractive 
operators and fixed point theory in generalized metric environments.
}

%{{{\abstract{In this paper, we introduce a new general framework, called \emph{perturbed extended $b$-metric spaces}, which extends the classical and extended $b$-metric structures by incorporating an explicit perturbation mapping.  This formulation is motivated by the observation that distance measurements in many analytical and applied contexts are often affected by intrinsic or external deviations that cannot be captured by the usual metric-type geometries. We also identify a meaningful specialization obtained when the control function is constant, leading to the notion of a \emph{perturbed $b$-metric space}, introduced here as a natural restriction of the general framework.We establish several fundamental properties of these spaces and develop a Banach-type fixed point theorem in the perturbed extended $b$-metric setting. Conditions ensuring the existence and uniqueness of fixed points are derived, together with convergence of Picard iterations. Illustrative examples are provided to demonstrate the necessity of the perturbation term and to highlight differences from standard $b$-metric and extended $b$-metric structures. The results presented here open new directions for the study of contractive operators in generalized metric environments.}




\maketitle
\section{Introduction and Preliminaries}
Fixed point theory stands at the heart of nonlinear analysis, offering a unified framework for studying nonlinear operators arising in functional equations, approximation methods, optimization, and dynamical systems. Banach’s contraction principle remains its classical foundation, and its clarity and flexibility have motivated numerous extensions to broader abstract settings. A complementary research direction addresses more flexible contractive conditions, using tools such as simulation functions \cite{khoja,simulation2,ext,control}, altering distances \cite{altering}, and control functions \cite{function1,function2}. This line of work also includes contractive schemes for cyclic, coupled, weak, or integral settings \cite{cyclic1,cyclic2,coupled}, together with techniques based on binary relations and graph frameworks \cite{relation1,relation2,relation3,relation4,relation5}. Studies on non-self mappings \cite{best1} have likewise expanded fixed point applications.

A natural direction in the evolution of fixed point theory is the modification of the underlying distance structure. Classical metric spaces often lack the flexibility required to describe phenomena influenced by uncertainty, irregularity, or non-standard interactions. To address these limitations, various generalized settings have been proposed, including fuzzy metric spaces \cite{kramosil,gv,surveyfuzzy}, intuitionistic fuzzy metric spaces \cite{intuimetric}, $b$-metric spaces \cite{bmetric1,bmetric2}, partial metric spaces \cite{partialmetric}, $G$-metric spaces \cite{Gmetric}, $\theta$-parametric metrics \cite{meparametric}, modular metrics \cite{modularmetric}, and probabilistic metrics \cite{probametric}. These frameworks relax classical axioms to capture behaviors that the standard metric cannot. 
Recent progress in this area was initiated by Jleli and Samet \cite{jleli}, who introduced perturbed metric spaces and established an extension of Banach’s principle, while Dobritoiu \cite{Dobri} applied this framework to obtain a uniqueness result for a nonlinear Fredholm integral equation. Further developments have also appeared, in \cite{Tijani}, the authors developed the notion of $C^{*}$-algebra-valued perturbed metric spaces, providing a broader framework that encompasses classical $C^{*}$-algebra-valued metric spaces. Păcurar et al. \cite{Pacurar} introduced the concept of mappings contracting perimeters of triangles in perturbed metric spaces.

At a broader level, perturbation concepts have gained attention as a means of incorporating external variations or intrinsic irregularities directly into the distance function itself. Introducing perturbations into the axioms of a metric structure allows the resulting space to reflect fluctuations, imperfections, or structural deviations that may arise in abstract or applied settings. This yields a richer class of generalized distance frameworks capable of capturing behaviors that classical metrics do not represent, thereby offering promising directions for studying convergence, stability, and iterative schemes in environments affected by noise or structural variability.

%Motivated by these developments, we introduce two new distance structures: perturbed $b$-metric spaces and perturbed extended $b$-metric spaces. In both settings, a perturbation term is embedded into the defining inequalities, producing distance behaviors that do not reduce to their classical counterparts. Several examples are provided to illustrate that these perturbed structures may fail properties that typically hold in the non-perturbed case.
The notion of a $b$-metric (or quasimetric) space arose from early attempts to generalize metric structures and develop more flexible analytical frameworks. While Bakhtin \cite{bmetric1} and later Czerwik \cite{bmetric2} played key roles in formalizing and promoting this concept within fixed point theory, Berinde and Păcurar \cite{vasile} highlighted that b-metric spaces are considerably older than is commonly acknowledged in recent literature. He also noted that the first fixed point theorem in this setting, the contraction principle for quasimetric spaces, appears to be due to Vulpe et al. \cite{Vulpe}. Working in b-metric spaces is topologically delicate, and the sustained effort over the past fifteen years to establish general fixed point results in this framework is therefore well justified. For concise overviews of the foundational developments in fixed point theory on $b$-metric spaces, we refer the reader to the surveys \cite{vasile,bmetricsurvey}.

\begin{definition}Let $X$ be a nonempty set and let $s \ge 1$ be a fixed real number. 
A mapping $d_{s} : X \times X \to [0,+\infty)$ is called a \emph{$b$-metric} 
if, for all $v,\omega,z \in X$, the following conditions hold:

\begin{description}
    \item[$(d_{s}1)$:] $d_{s}(v,\omega)=0$ if and only if $v=\omega$;
    \item[$(d_{s}2)$:] $d_{s}(v,\omega)=d_{s}(\omega,v)$;
    \item[$(d_{s}3)$:] $d_{s}(v,z) \le s\,[\,d_{s}(v,\omega)+d_{s}(\omega,z)\,]$.
\end{description}
In this case, the pair $(X,d_{s})$ is called a \emph{$b$-metric space}.
\end{definition}
\vspace{0.2cm}
\begin{remark}
For $s = 1$, the $b$-metric becomes a standard metric, and convergence, Cauchy sequences, and completeness are understood exactly as in metric spaces.
\end{remark}
\vspace{0.2cm}
\begin{remark}
It is worth noting that a $b$-metric does not necessarily define a continuous 
distance functional. As a consequence, an extended $b$-metric need not be 
continuous either. The following example illustrates this fact.
\end{remark}
\vspace{0.2cm}
\begin{example}\cite{Nawab}
Let $X=\mathbb{N}\,\cup\,\{+\infty\},$
and define $d_{s} : X\times X \to \mathbb{R}$ by
\[
d_{s}(i,j)=
\begin{cases}
0, & i=j,\\[2mm]
\bigl|\tfrac{1}{i}-\tfrac{1}{j}\bigr|, & i,j \text{ even or } ij=+\infty,\\[2mm]
5, & i,j \text{ odd and } i\neq j,\\[2mm]
2, & \text{otherwise}.
\end{cases}
\]
One verifies that $(X,d_{s})$ forms a $b$-metric space with coefficient $s=3$, 
yet the mapping $d_{s}$ fails to be continuous.
\end{example}
\vspace{0.2cm}

Kamran et al.~\cite{kamran} introduced a broader class of generalized metric spaces, 
referred to as \emph{extended $b$-metric spaces}. The idea is to allow the controlling 
coefficient in the $b$-metric inequality to vary with the points involved, thus 
providing a more flexible structure.
\begin{definition}
Let $X$ be a nonempty set and let $\zeta : X \times X \to [1,+\infty)$ be a given function.  
A mapping $d_{\zeta} : X \times X \to [0,+\infty)$ is called an \emph{extended $b$-metric} 
if, for all $v,\omega,z \in X$, the following conditions hold:
\begin{description}
    \item[$(d_{\zeta}1)$:] $d_{\zeta}(v,\omega)=0$ if and only if $v=\omega$;
    \item[$(d_{\zeta}2$):] $d_{\zeta}(v,\omega)=d_{\zeta}(\omega,v)$;
    \item[$(d_{\zeta}3$):] $d_{\zeta}(v,z) \le \zeta(v,z)\,\big(d_{\zeta}(v,\omega)+d_{\zeta}(\omega,z)\big)$.
\end{description}
In this case, the pair $(X,d_{\zeta})$ is called an \emph{extended $b$-metric space}.
\end{definition}
\vspace{0.2cm}
\begin{remark}
If one slightly adjusts the above definition by allowing the control function 
to depend on three variables instead of two, that is, by replacing 
$\zeta : X \times X \to [1,+\infty)$ with a function $\zeta : X \times X \times X \to [1,+\infty),$ 
then a different form of extended $b$-metric arises.  
This three–point control version has been investigated in a work of 
Aydi et al.~\cite{Aydi}, where the extended inequality involves 
$\zeta(v,\omega,z)$ rather than $\zeta(v,z)$. In Particular, if $\zeta(v,\omega,z)=\zeta(v,z)$ for all $v,\omega,z\in X$, then the 
three-point form reduces to the usual extended $b$-metric.
\end{remark}
\vspace{0.2cm}
\begin{example}
Let $X=\{1,2,3\}$. Define the mappings 
$\zeta : X\times X \to \mathbb{R}^{+}$ and 
$d_{\zeta} : X\times X \to \mathbb{R}^{+}$ as follows. 
For all $v,\omega \in X$, set
\[
\zeta(v,\omega)=1+v+\omega.
\]
where
\[
d_{\zeta}(1,1)=d_{\zeta}(2,2)=d_{\zeta}(3,3)=0,
\]
\[
d_{\zeta}(1,2)=d_{\zeta}(2,1)=80,\qquad
d_{\zeta}(1,3)=d_{\zeta}(3,1)=1000,\qquad
d_{\zeta}(2,3)=d_{\zeta}(3,2)=600.
\]
Then, as proved in \cite{kamran}, the structure $(X, d_{\zeta})$ is an extended $b$-metric space.
\end{example}
\begin{example} \cite{Samreen} \label{ex:extended-bmetric}
Let $X=\{1,2,3,\ldots\}$. Define the control function 
\[
\zeta:X\times X\to[1,+\infty), \qquad 
\zeta(v,\omega)=
\begin{cases}
|v-\omega|^{3}, & v\neq \omega,\\[2mm]
1, & v=\omega,
\end{cases}
\]
and define the mapping $d_{\zeta}:X\times X\to[0,+\infty)$ by
\[
d_{\zeta}(v,\omega)=(v-\omega)^{4}.
\]
Then the pair $(X,d_{\zeta})$ forms an extended $b$-metric space with respect to the 
control function $\zeta$.
\end{example}

\vspace{0.2cm}
\begin{lemma}\cite{kamran}
Let $(X,d_{\zeta})$ be an extended $b$-metric space. If the mapping 
$d_{\zeta}:X\times X\to[0,+\infty)$ is continuous, then every convergent sequence 
in $X$ admits a unique limit.
\end{lemma}

\section{Main results}
In practical and theoretical settings, the evaluation of distances often carries 
unavoidable inaccuracies arising from various sources, such as measurement 
limitations or intrinsic fluctuations in the underlying system. Although these 
deviations may, at first sight, appear negligible, their cumulative effect can 
significantly modify the geometric behaviour of the space. To incorporate such 
perturbative phenomena into a rigorous analytic framework, we work within the 
structure of a \emph{perturbed extended $b$-metric} $(X,\mathcal{D}_{\zeta},\hbar)$,  
in which the observed distance $\mathcal{D}_{\zeta}$ differs from its exact 
counterpart by an explicit perturbation term governed by a control function 
$\zeta$. In this section, we introduce the core definitions, structural notions, 
and preliminary remarks that will support the fixed point theory developed in 
the subsequent results.

\vspace{0.2cm}
% ============================
% Definition: Perturbed Extended b-Metric
% ============================
\begin{definition}\label{prensentdef}
Let $\mathcal{D}_{\zeta},\, \hbar : X \times X \to [0,+\infty)$ be two given mappings, 
and let 
\[
\zeta : X \times X \to [1,+\infty)
\]
be a control function. We say that $\mathcal{D}_{\zeta}$ is a 
\textit{perturbed extended $b$-metric} on $X$ with respect to $\hbar$ if 
\[
\mathcal{D}_{\zeta} - \hbar : X \times X \to \mathbb{R},
\qquad
(v,\omega) \longmapsto \mathcal{D}_{\zeta}(v,\omega) - \hbar(v,\omega),
\]
is an extended $b$-metric on $X$. That is, for all $v,\omega,z \in X$:
\begin{description}
    \item[$(\mathcal{D}_{\zeta}1):$] $(\mathcal{D}_{\zeta} - \hbar)(v,\omega) \ge 0$;
    \item[$(\mathcal{D}_{\zeta}2):$] $(\mathcal{D}_{\zeta} - \hbar)(v,\omega) = 0$ if and only if $v=\omega$;
    \item[$(\mathcal{D}_{\zeta}3):$] $(\mathcal{D}_{\zeta} - \hbar)(v,\omega)
           = (\mathcal{D}_{\zeta} - \hbar)(\omega,v)$;
    \item[$(\mathcal{D}_{\zeta}4):$] \(
    (\mathcal{D}_{\zeta} - \hbar)(v,z)
    \le 
    \zeta(v,z)\Big[
    (\mathcal{D}_{\zeta} - \hbar)(v,\omega)
    +
    (\mathcal{D}_{\zeta} - \hbar)(\omega,z)
    \Big].
    \)
\end{description}

We call $\hbar$ the \textit{perturbation mapping}, \(
d_{\zeta} = \mathcal{D}_{\zeta} - \hbar
\)
the \textit{exact extended $b$-metric}, 
and the triple $(X,\mathcal{D}_{\zeta},\hbar)$ a 
\textit{perturbed extended $b$-metric space}.
\end{definition}
\vspace{0.2cm}
\begin{remark}
\begin{description}\item[]
    \item[i)] If the above definition is modified by allowing the control function 
to depend on three variables, that is, by replacing 
$\zeta : X \times X \to [1,+\infty)$ with a mapping 
$\zeta : X \times X \times X \to [1,+\infty)$ in $(\mathcal{D}_{\zeta}4)$, 
then one obtains a new three-point variant of a perturbed extended $b$-metric.  
This yields a broader concept that generalizes the three-point extended 
$b$-metric introduced by Aydi et al.~\cite{Aydi}.

    \item[ii)] If, moreover, $\zeta(v,\omega,z)=\zeta(v,z)$ for all $v,\omega,z\in X$, 
then this new three-point version reduces to the present definition \ref{prensentdef}.
\end{description}

\end{remark}
\vspace{0.2cm}

\begin{remark}
In a perturbed extended $b$-metric space, the presence of the perturbation term 
$\hbar$ may cause $\mathcal{D}_{\zeta}$ to lose the usual extended $b$-metric 
behaviour and prevent $\mathcal{D}_{\zeta}$ from satisfying the standard 
extended $b$-metric axioms, such as the condition $\mathcal{D}_{\zeta}(v,v)=0$ 
or the extended $b$-triangle inequality. The next example confirms that a 
perturbed extended $b$-metric on $X$ is not necessarily an extended $b$-metric 
on $X$.
\end{remark}
\vspace{0.2cm}
% ============================
% Example: Perturbed Extended b-Metric on C([a,b],R)
% ============================
\begin{example}
Consider the mapping \(
\mathcal{D}_{\zeta} : X \times X \longrightarrow [0,++\infty),
\qquad X = C([a,b],\mathbb{R}),
\)
defined by
\[
\mathcal{D}_{\zeta}(v,\omega)
=
\sup_{t \in [a,b]} |v(t) - \omega(t)|^{2}
+
\sup_{t \in [a,b]} \left( |v(t)|^{2} + |\omega(t)|^{2} \right),
\qquad \text{for all } v,\omega \in X.
\]

Then $\mathcal{D}_{\zeta}$ constitutes a perturbed extended $b$-metric on $X$ 
with respect to the perturbation 
\[
\hbar : X \times X \longrightarrow [0,++\infty),
\]
given by
\[
\hbar(v,\omega)
=
\sup_{t \in [a,b]}
\left( |v(t)|^{2} + |\omega(t)|^{2} \right),
\qquad \text{for all } v,\omega \in X.
\]

Accordingly, the corresponding exact extended $b$-metric \(
d_{\zeta} : X \times X \longrightarrow [0,++\infty)
\)
is expressed as
\[
d_{\zeta}(v,\omega)
=
\sup_{t \in [a,b]} |v(t) - \omega(t)|^{2},
\qquad \text{for all } v,\omega \in X.
\]

It is worth noting that $\mathcal{D}_{\zeta}$ is not an extended $b$-metric on $X$, since
\[
\mathcal{D}_{\zeta}(v,v)
=
\sup_{t \in [a,b]} \left( |v(t)|^{2} + |v(t)|^{2} \right)
\neq 0,
\qquad \text{for any } v \not\equiv 0.
\]
Thus $\mathcal{D}_{\zeta}(v,v) > 0$, showing that the perturbation $\hbar$ is essential.
\end{example}

\vspace{0.2cm}
\begin{remark}
Let $(X,\mathcal{D}_{\zeta},\hbar)$ be a perturbed extended $b$-metric space with 
control function $\zeta : X\times X \to [1,+\infty)$ and associated exact extended 
$b$-metric $d_{\zeta}=\mathcal{D}_{\zeta}-\hbar$.

If the control function is constant, that is,
\[
\zeta(v,\omega)=s\ge 1, \qquad \text{for all } v,\omega\in X,
\]
then the extended $b$-metric inequality
\[
d_{\zeta}(v,z)\le 
\zeta(v,z)\,\big(d_{\zeta}(v,\omega)+d_{\zeta}(\omega,z)\big)
\]
reduces to the classical $b$-metric condition
\[
d_{\zeta}(v,z)\le 
s\,[\,d_{\zeta}(v,\omega)+d_{\zeta}(\omega,z)\,].
\]
Hence, in this case, $d_{\zeta}$ becomes a $b$-metric on $X$ with coefficient $s$.

This observation motivates the introduction of a specific terminology for the 
constant-control case. When the control function reduces to a fixed parameter 
$s\ge 1$, the triple $(X,\mathcal{D}_{s},\hbar)$ becomes a precise specialization 
of the perturbed extended $b$-metric structure. We therefore refer to such a 
configuration as a \emph{perturbed $b$-metric space}. This framework emerges as a natural and structurally coherent restriction of the 
general perturbed extended $b$-metric concept developed in this work, and it 
naturally leads to the dedicated notion introduced in the following definition.
\end{remark}
\vspace{0.2cm}
\begin{definition}
Let $\mathcal{D}_{s},\, \hbar : X \times X \to [0,+\infty)$ be two given mappings, 
and let $s \ge 1$ be a fixed constant.  
We say that $\mathcal{D}_{s}$ is a \textit{perturbed $b$-metric} on $X$ with 
respect to $\hbar$ if 
\[
\mathcal{D}_{s} - \hbar : X \times X \to \mathbb{R},
\qquad
(v,\omega) \longmapsto \mathcal{D}_{s}(v,\omega) - \hbar(v,\omega),
\]
defines a $b$-metric on $X$ with coefficient $s$.  
That is, for all $v,\omega,z \in X$:

\begin{description}
    \item[$(\mathcal{D}_{s}1):$] $(\mathcal{D}_{s} - \hbar)(v,\omega) \ge 0$;
    \item[$(\mathcal{D}_{s}2):$] $(\mathcal{D}_{s} - \hbar)(v,\omega) = 0$ if and only if $v=\omega$;
    \item[$(\mathcal{D}_{s}3):$] $(\mathcal{D}_{s} - \hbar)(v,\omega)
           = (\mathcal{D}_{s} - \hbar)(\omega,v)$;
    \item[$(\mathcal{D}_{s}4):$] \(
    (\mathcal{D}_{s} - \hbar)(v,z)
    \le 
    s\Big[
      (\mathcal{D}_{s} - \hbar)(v,\omega)
      +
      (\mathcal{D}_{s} - \hbar)(\omega,z)
    \Big].
    \)
\end{description}

In this case, the expression $d_{s} = \mathcal{D}_{s} - \hbar$ is called the \textit{exact $b$-metric}, and the triple 
$(X,\mathcal{D}_{s},\hbar)$ is referred to as a 
\textit{perturbed $b$-metric space}.
\end{definition}
\vspace{0.2cm}
\begin{remark}
When $s=1$, the exact mapping $d_{1}=\mathcal{D}_{1}-\hbar$ is a usual metric, and $(X,\mathcal{D}_{1},\hbar)$ becomes a perturbed metric space. Hence, the notion above reduces to the classical concept of a perturbed metric space established in \cite{jleli}.
\end{remark}
\vspace{0.2cm}
\begin{remark}
Note that the perturbation $\hbar$ may cause $\mathcal{D}_{s}$ to fail the 
defining conditions of a $b$-metric. The next example demonstrates that a perturbed 
$b$-metric on $X$ is not necessarily a $b$-metric.
\end{remark}
\vspace{0.2cm}

\begin{example}[Perturbed $b$-metric on $\ell^{p}(\mathbb{R})$, $0<p<1$]
Consider the mapping 
\[
\mathcal{D}_{s}:\ell^{p}(\mathbb{R})\times \ell^{p}(\mathbb{R})\rightarrow [0,++\infty)
\]
defined by
\[
\mathcal{D}_{s}(v,\omega)
=
\left( \sum_{n=1}^{+\infty} |v_{n}-\omega_{n}|^{p} \right)^{1/p}
+
\sum_{n=1}^{+\infty} \big( |v_{n}|^{p} + |\omega_{n}|^{p} \big),
\qquad v,\omega \in \ell^{p}(\mathbb{R}),
\]
where 
\[
\ell^{p}(\mathbb{R})
=
\Big\{
v=(v_{n})_{n\ge 1} : \sum_{n=1}^{+\infty} |v_{n}|^{p} < +\infty
\Big\},
\qquad 0<p<1.
\]

The mapping $\mathcal{D}_{s}$ defines a perturbed $b$-metric on $\ell^{p}(\mathbb{R})$ 
with respect to the perturbation 
\[
\hbar(v,\omega)
=
\sum_{n=1}^{+\infty} \big( |v_{n}|^{p} + |\omega_{n}|^{p} \big),
\qquad 
v,\omega \in \ell^{p}(\mathbb{R}).
\]

Accordingly, the corresponding exact $b$-metric 
$d_{s}:\ell^{p}(\mathbb{R})\times \ell^{p}(\mathbb{R})\rightarrow [0,++\infty)$ 
is given by
\[
d_{s}(v,\omega)
=
\mathcal{D}_{s}(v,\omega) - \hbar(v,\omega)
=
\left( \sum_{n=1}^{+\infty} |v_{n}-\omega_{n}|^{p} \right)^{1/p},
\qquad v,\omega \in \ell^{p}(\mathbb{R}).
\]
It is known that $d_{s}$ is a $b$-metric on $\ell^{p}(\mathbb{R})$ with coefficient 
$s = 2^{1/p}$.

Furthermore, $\mathcal{D}_{s}$ itself is not a $b$-metric on $\ell^{p}(\mathbb{R})$, since
\[
\mathcal{D}_{s}(v,v)
=
d_{s}(v,v) + \hbar(v,v)
=
\sum_{n=1}^{+\infty} 2|v_{n}|^{p} > 0
\quad\text{for any nonzero } v\in \ell^{p}(\mathbb{R}).
\]
\end{example}

\begin{proposition}
Let $(X,\mathcal{D}_{\zeta},\hbar)$ be a perturbed extended $b$-metric space, and
let $d_{\zeta} = \mathcal{D}_{\zeta} - \hbar$
be the associated exact extended $b$-metric. Suppose that the mappings
\[
\mathcal{D}_{\zeta},\, \hbar : X \times X \to [0,+\infty)
\]
are continuous. Then $d_{\zeta}$ is a continuous extended $b$-metric on $X$. 
In particular, since a continuous extended $b$-metric ensures uniqueness of
the limit of any convergent sequence (see Lemma~1. \cite{kamran}), every convergent sequence 
in $(X,d_{\zeta})$ has a unique limit.
\end{proposition}
\vspace{0.3cm}

We now initiate some topological notions in perturbed extended b-metric spaces with the parallel concepts in perturbed b-metric spaces taken implicitly.
\vspace{0.3cm}
\begin{definition}
Let $(X,\mathcal{D}_{\zeta},\hbar)$ be a perturbed extended $b$-metric space, and let 
$(X,d_{\zeta})$ denote the associated exact extended $b$-metric space given by \(
d_{\zeta}(x,y)=\mathcal{D}_{\zeta}(x,y)-\hbar(x,y).
\) Let $\{v_n\}_{n\in\mathbb{N}}$ be a sequence in $X$, and let $T:X\to X$ be a mapping.
\begin{enumerate}[(i)]
\item We say that $\{v_n\}$ is \emph{perturbed convergent in} $(X,\mathcal{D}_{\zeta},\hbar)$ 
if there exists $v\in X$ such that \(d_{\zeta}(v_n,v)\longrightarrow 0 \quad \text{as } n\to+\infty,
\) that is, $\{v_n\}$ converges to $v$ in the exact extended $b$-metric space $(X,d_{\zeta})$.
\item The sequence $\{v_n\}$ is called a \emph{perturbed Cauchy sequence in} 
$(X,\mathcal{D}_{\zeta},\hbar)$ if \(
d_{\zeta}(v_n,v_m)\longrightarrow 0 \quad \text{as } n,m\to+\infty,
\) i.e., $\{v_n\}$ is Cauchy in the exact extended b-metric space $(X,d_{\zeta})$.

\item The space $(X,\mathcal{D}_{\zeta},\hbar)$ is said to be 
\emph{complete in the perturbed extended $b$-metric sense} whenever the exact space 
$(X,d_{\zeta})$ is complete, or equivalently, whenever every perturbed Cauchy sequence 
in $(X,\mathcal{D}_{\zeta},\hbar)$ is perturbed convergent.

\item A mapping $T:X\to X$ is called \emph{perturbed continuous} if 
whenever $v_n \to v$ in the exact space $(X,d_{\zeta})$, one also has \(
d_{\zeta}(Tv_n, Tv)\longrightarrow 0,
\) that is, $T$ preserves perturbed convergence in $(X,\mathcal{D}_{\zeta},\hbar)$.
\end{enumerate}
\end{definition}

We next establish a Banach-type fixed point theorem formulated in the newly introduced setting of perturbed extended $b$-metric spaces.
%For $T:X\to X$ and $v_{0}\in X$, we denote the generated orbit by\[\mathcal{O}(v_{0})=\{\,v_{0},\,Tv_{0},\,T^{2}v_{0},\ldots\,\}.\]

\vspace{0.3cm}

\begin{theorem}\label{3.1} 
Let $(X,\mathcal{D}_{\zeta},\hbar)$ be a complete perturbed extended $b$-metric space. Assume that the associated exact extended b-metric is a continuous function. Let $T : X \to X$ be a mapping satisfying
\begin{equation}
\mathcal{D}_{\zeta}(T v,T \omega) 
\le c.\, \mathcal{D}_{\zeta}(v,\omega),
\qquad \text{for all } v,\omega \in X,
\end{equation}
for some constant $c \in [0,1)$. 

Moreover, suppose that for each $v_{0} \in X$, the sequence 
$v_{n} = T^{n}v_{0}$, $n = 1,2,\dots$, satisfies
\[
\lim_{n,m \to +\infty} \zeta(v_{n},v_{m}) < \frac{1}{c}.
\]
Then $T$ has a unique fixed point $\vartheta \in X$. Furthermore, for every 
$v \in X$, the sequence $(T^{n}v)$ converges to $\vartheta$ with respect to $\mathcal{D}_{\zeta}$.
\end{theorem}

\begin{proof}
Let $v_{0}\in X$ be arbitrary and construct the iterative sequence 
$\{v_{n}\}$ by
\[
v_{0}, \qquad v_{1}=T v_{0}, \qquad v_{2}=T v_{1}=T^{2}v_{0},\quad \dots,\quad 
v_{n}=T^{n}v_{0},\ \dots
\]

From the contractive condition
\[
\mathcal{D}_{\zeta}(T v,T \omega) \le c\,\mathcal{D}_{\zeta}(v,\omega)
\qquad \text{for all } v,\omega\in X,
\]
we deduce, for $n=0$,
\[
\mathcal{D}_{\zeta}(v_{1},v_{2})
= \mathcal{D}_{\zeta}(T v_{0},T v_{1})
\le c\,\mathcal{D}_{\zeta}(v_{0},v_{1}).
\]

Assume now that for some $n\ge0$ we have
\[
\mathcal{D}_{\zeta}(v_{n},v_{n+1})
\le c^{n}\,\mathcal{D}_{\zeta}(v_{0},v_{1}).
\]
Using the contractive condition with $v=v_{n}$ and $\omega=v_{n+1}$, we obtain
\[
\mathcal{D}_{\zeta}(v_{n+1},v_{n+2})
= \mathcal{D}_{\zeta}(T v_{n},T v_{n+1})
\le c\, \mathcal{D}_{\zeta}(v_{n},v_{n+1})
\le c^{\,n+1}\,\mathcal{D}_{\zeta}(v_{0},v_{1}).
\]

Thus, by induction on $n$, we conclude that
\begin{equation}\label{Dz-iter}
\mathcal{D}_{\zeta}(v_{n},v_{n+1})
\le c^{\,n}.\Im ~~\text{ for all } n\in\mathbb{N},
\end{equation}
where $\Im=\mathcal{D}_{\zeta}(v_{0},v_{1}).$

Let \(
d_{\zeta}(u,w)=\mathcal{D}_{\zeta}(u,w)-\hbar(u,w)
\)
be the exact extended $b$-metric associated with $(X,\mathcal{D}_{\zeta},\hbar)$.  By the definition of $d_{\zeta}$, we may rewrite \eqref{Dz-iter} as
\[
d_{\zeta}(v_n,v_{n+1}) + \hbar(v_n,v_{n+1})
= \mathcal{D}_{\zeta}(v_n,v_{n+1})
\le c^{\,n}\,\mathcal{D}_{\zeta}(v_0,v_1), \qquad n\in\mathbb{N}.
\]

Since
\[
d_{\zeta}(v_n,v_{n+1})
\le d_{\zeta}(v_n,v_{n+1})+\hbar(v_n,v_{n+1})
= \mathcal{D}_{\zeta}(v_n,v_{n+1}),
\]
we deduce from \eqref{Dz-iter} that
\begin{equation}\label{dz-final}
d_{\zeta}(v_n,v_{n+1})
\le c^{\,n}.\Im,
\qquad n\in\mathbb{N}.
\end{equation}
%A standard argument applied to \eqref{dz-final} shows that the sequence 
%$\{v_n\}$ is Cauchy in the exact extended $b$-metric space $(X,d_{\zeta})$.  
%Since $(X,\mathcal{D}_{\zeta},\hbar)$ is complete (i.e., $(X,d_{\zeta})$ is complete), 
%there exists a point $\vartheta\in X$ such that
%\[
%\lim_{n\to+\infty} d_{\zeta}(v_n,\vartheta)=0.
%\]
%Thus the sequence $\{v_n\}$ converges to $\vartheta$ in the metric space $(X,d_{\zeta})$.
Using successively the extended $b$-metric inequality, for $m>n$ we have
\[
\begin{aligned}
d_{\zeta}(v_n , v_m)
&\le \zeta(v_n , v_m)\, d_{\zeta}(v_n , v_{n+1})
   + \zeta(v_n , v_m)\zeta(v_{n+1},v_m)\, d_{\zeta}(v_{n+1},v_{n+2})  \\
&\quad +\; \zeta(v_n , v_m)\zeta(v_{n+1},v_m)\zeta(v_{n+2},v_m)\,
          d_{\zeta}(v_{n+2},v_{n+3}) + \cdots \\
&\quad +\; \zeta(v_n , v_m)\zeta(v_{n+1},v_m)\cdots \zeta(v_{m-1},v_m)\,
          d_{\zeta}(v_{m-1},v_m).
\end{aligned}
\]

From the estimate
\[
d_{\zeta}(v_j , v_{j+1}) \le c^{\,j}\,\Im, \qquad j\in\mathbb{N},
\]
we deduce
\[
\begin{aligned}
d_{\zeta}(v_n , v_m)
&\le \Im \Big[
c^{n}\prod_{i=1}^{n}\zeta(v_i,v_m)
+ c^{\,n+1}\prod_{i=1}^{n+1}\zeta(v_i,v_m)
+ \cdots  \\
&\qquad\qquad 
+ c^{\,m-1}\prod_{i=1}^{m-1}\zeta(v_i,v_m)
\Big].
\end{aligned}
\]

Since 
\[
\lim_{n,m\to+\infty} \zeta(v_{n+1},v_m)\,c < 1,
\]
the series
\[
\sum_{j=1}^{+\infty} c^{\,j}\prod_{i=1}^{j}\zeta(v_i,v_m)
\]
converges by the ratio test for each fixed $m\in\mathbb{N}$.  
Define
\[
\daleth=\sum_{j=1}^{+\infty} c^{\,j}\prod_{i=1}^{j}\zeta(v_i,v_m),
\qquad
\daleth_n=\sum_{j=1}^{n} c^{\,j}\prod_{i=1}^{j}\zeta(v_i,v_m).
\]

Thus, for $m>n$, the above inequality gives
\[
d_{\zeta}(v_n , v_m)
\le \Im\,\big(\daleth_{m-1}-\daleth_n\big).
\]

Letting $n\to+\infty$, we conclude that 
\[
d_{\zeta}(v_n,v_m)\longrightarrow 0,
\]
hence the sequence $\{v_n\}$ is Cauchy in the exact extended $b$-metric space $(X,d_{\zeta})$. In particular, by definition of the perturbed structure, this means that $\{v_n\}$ is a perturbed Cauchy sequence in $(X,\mathcal{D}_{\zeta},\hbar)$. Since $(X,\mathcal{D}_{\zeta},\hbar)$ is complete in the perturbed sense, and 
$\{v_n\}$ has been shown to be Cauchy in $(X,d_{\zeta})$, there exists a point 
$\vartheta \in X$ such that
\begin{equation}\label{limit-v}
\lim_{n \to +\infty} d_{\zeta}(v_n,\vartheta)=0.
\end{equation}

\medskip
We now verify that $\vartheta$ is a fixed point of $T$.  
Using the perturbed continuity of $T$ with respect to the exact extended $b$-metric $d_{\zeta}$, 
the convergence in \eqref{limit-v} implies
\[
\lim_{n\to+\infty} d_{\zeta}(T v_n, T\vartheta)=0.
\]
Since $T v_n = v_{n+1}$, we obtain
\begin{equation}\label{limit-T}
\lim_{n\to+\infty} d_{\zeta}(v_{n+1}, T\vartheta)=0.
\end{equation}
Because $d_{\zeta}$ is a genuine extended $b$-metric on $X$, limits are unique; 
comparing \eqref{limit-v} and \eqref{limit-T} yields
\[
T\vartheta = \vartheta.
\]
Thus $\vartheta$ is indeed a fixed point of $T$.

\medskip
We next show that this fixed point is the only one.  
Assume, to the contrary, that $p,q\in X$ are two distinct fixed points of $T$.  
Applying the contractive condition to $(p,q)$ gives
\[
\mathcal{D}_{\zeta}(p,q)
= \mathcal{D}_{\zeta}(Tp,Tq)
\le c\, \mathcal{D}_{\zeta}(p,q).
\]
Writing $\mathcal{D}_{\zeta}=d_{\zeta}+\hbar$, we get
\[
d_{\zeta}(p,q)+\hbar(p,q)
\le c\,\big( d_{\zeta}(p,q)+\hbar(p,q) \big).
\]
Since $p\neq q$, the quantity $d_{\zeta}(p,q)+\hbar(p,q)$ is strictly positive, and 
the above inequality forces $c\ge 1$, contradicting the assumption $0\le c<1$.
Hence, the fixed point $\vartheta$ is unique. This completes the proof of the theorem.
\end{proof}

\begin{corollary} \cite{kamran}
Let $(X,d_{\zeta})$ be a complete extended $b$-metric space such that 
$d_{\zeta}$ is a continuous functional. 
Let $T : X \to X$ satisfy
\begin{equation}\label{eq:contr-ext-b}
d_{\zeta}(Tv, T\omega) \le c\, d_{\zeta}(v, \omega)
\quad \text{for all } v, \omega \in X,
\end{equation}
where $c \in [0,1)$ is such that, for each $v_{0} \in X$,
\[
\lim_{n,m \to +\infty} \zeta(v_n, v_m) < \frac{1}{c},
\]
here $v_n = T^{n} v_0$, $n = 1,2,\dots$.

Then $T$ has precisely one fixed point $v^{\ast} \in X$. Moreover, for each 
$\omega \in X$, we have $T^{n}\omega \to v^{\ast}$ as $n \to +\infty$.
\end{corollary}
\begin{proof}
Define 
$\hbar(v,\omega) = 0$ for all $v,\omega \in X$ and let $\mathcal{D}_{\zeta} = d_{\zeta}$. Then 
$(X,\mathcal{D}_{\zeta},\hbar)$ is a perturbed extended $b$-metric space. 
Furthermore, by \eqref{eq:contr-ext-b}, $T$ is continuous with respect to the 
exact extended $b$-metric, and the contractive condition 
of Theorem \ref{3.1} is satisfied. Hence Theorem \ref{3.1} applies, and the conclusion follows.
\end{proof}
\section{Discussion: Perturbative Extensions to $S_b$-Metric Spaces}
Finally, we observe that the perturbation methodology developed in this work 
naturally extends to other multi-point distance structures. In particular, the 
same decomposition technique allows the introduction of a \emph{perturbed 
extended $S_{b}$-metric space}. Recall that, following Mlaiki \cite{Mlaiki} and 
Souayah et al.\ \cite{Souayah}, a mapping $S_{\zeta}:X^{3}\to[0,+\infty)$ is 
said to define an extended $S_{b}$-metric space $(X,S_{\zeta})$ if there exists 
a control function $\zeta:X^{3}\to[1,+\infty)$ such that, for all 
$v,\omega,z,t\in X$,
\[
\begin{aligned}
&\text{(i)}\quad S_{\zeta}(v,\omega,z)=0 \iff v=\omega=z,\\[1mm]
&\text{(ii)}\quad 
S_{\zeta}(v,\omega,z)
   \le 
\zeta(v,\omega,z)\,\big(
    S_{\zeta}(v,v,t)
  + S_{\zeta}(\omega,\omega,t)
  + S_{\zeta}(z,z,t)
\big).
\end{aligned}
\]

When the control function is constant, i.e.\ $\zeta(v,\omega,z)=s\ge 1$, one 
obtains the $S_{b}$-metric introduced in \cite{Mlaiki,Souayah}. Moreover, taking 
$s=1$ yields the classical $S$-metric of Sedghi et al.\ \cite{Sedghi}.

Motivated by the philosophy developed in the present paper, we initiate 
the formulation of a perturbed extended $S_{b}$-metric space within this 
perturbative framework. 

\begin{definition}
Let $\mathcal{S},\hbar_{S}:X^{3}\to[0,+\infty)$ be two mappings and let 
$\zeta:X^{3}\to[1,+\infty)$ be a control function.  
We say that $\mathcal{S}$ is a \emph{perturbed extended $S_{b}$-metric} on $X$
with respect to $\hbar_{S}$ if the mapping
\[
S_{b}(v,\omega,z)
    = \mathcal{S}(v,\omega,z) - \hbar_{S}(v,\omega,z)
\]
defines an extended $S_{b}$-metric on $X$ in the sense of \cite{Mlaiki}.
\end{definition}

In this setting, given a pair $(\mathcal{S},\hbar_{S})$, we define the 
exact component
\[
S_{b}(v,\omega,z)
    = \mathcal{S}(v,\omega,z)
    - \hbar_{S}(v,\omega,z),
\]
and require that $S_{b}$ satisfies the axioms above. When the control function 
$\zeta$ is constant, this leads naturally to a \emph{perturbed $S_{b}$-metric}, 
and the case $s=1$ yields a corresponding \emph{perturbed $S$-metric}.

This hierarchy parallels the structure developed here for perturbed extended 
$b$-metrics and opens a natural avenue for further study in multi-point 
perturbative geometries.

\section*{Conclusion}

We introduced perturbed extended $b$-metric spaces as a flexible framework in 
which distance measurements incorporate both a perturbation term and a control 
function. This structure accommodates deviations in the underlying geometry 
while preserving a setting suitable for fixed point analysis. A Banach-type 
contraction theorem was established, and the behavior of Picard iterations was 
clarified under the perturbative assumptions. The examples provided highlight 
how perturbations can meaningfully affect the geometry associated with the 
exact distance.

The perturbative viewpoint extends naturally beyond two-point distances. 
Multi-point models, including the $S$-metric, its extended forms, fit coherently into the same conceptual scheme. 
This suggests a broader landscape of perturbation-based distance theories.

Future work may explore additional contractive conditions, stability 
properties of iterative processes, and adaptations of the framework to other 
generalized metric structures.

\end{document}